\documentclass[a4paper,11pt]{article}
\usepackage{algorithmic}
\usepackage{algorithm}
\usepackage{amsmath}
\usepackage{amsthm}
\usepackage{amssymb}
\usepackage{vmargin}
\usepackage{setspace}

\newcommand{\QNUM}{D}

\newcommand{\HYBINP}{\mathcal{U}}

\newcommand{\LSS}{DTLSS}
\newcommand{\BLSS}{DTLSS\ }
\newcommand{\SLSS}{DTLSSs}
\newcommand{\BSLSS}{DTLSSs\ }
\newcommand{\GCR}{\textbf{GCR}}

\newcommand{\IM}{\mathrm{Im}}
\newcommand{\Rank}{\mathrm{rank}\mbox{ } }
\newtheorem{Theorem}{Theorem}
\newtheorem{Lemma}{Lemma}
\newtheorem{Definition}{Definition}
\newtheorem{Remark}{Remark}
\newtheorem{Notation}{Notation}

\renewcommand{\paragraph}[1]{\smallskip\noindent\textbf{#1.} }
\renewcommand{\Re}{\mathbb{R}}

\title{On the notion of persistence of excitation for linear switched systems}
\date{}
\author{Mih\'aly Petreczky$^\dag$ and Laurent Bako$^\ddag$ \\
    $^{\dag}$Maastricht University,  \\
   P.O. Box 616, 6200 MD Maastricht, The Netherlands\\
   \texttt{M.Petreczky@maastrichtuniversity.nl} \\
  $^{\ddag}$Univ Lille Nord de France, \\ F-59000 Lille, France,\\
    and EMDouai, IA, F-59500 Douai, France,  \\
      \texttt{laurent.bako@mines-douai.fr.} 
}

\begin{document}
\maketitle

\setstretch{1.25}

\begin{abstract}
 The paper formulates the concept of persistence of excitation for
 discrete-time linear switched systems, and provides sufficient
 conditions for an input signal to be persistently exciting.
 Persistence of excitation is formulated as a property of the
 input signal, and it is not tied to any specific identification
 algorithm. The results of the paper rely on realization theory
 and on the notion of Markov-parameters for
 linear switched systems.
\end{abstract}

\section{Introduction}
 The paper formulates the concept of persistence of excitation for
 \emph{discrete-time linear switched systems} (abbreviated by \SLSS).
 \BSLSS are one of the simplest and best studied
 classes of hybrid systems, \cite{Sun:Book}. A \LSS\ is a
 discrete-time switched system, such that the continuous
 sub-system associated with each discrete state is linear.
 The switching signal is viewed as an external input, and all
 linear systems live on the same input-output- and state-space.

 We define persistence of excitation for input signals. More 
 precisely, we will call an input signal persistently exciting
 for an input-output map, if the response of the input-output map
 to that particular input determines the input-output map
 uniquely. In other words, the knowledge of the output response
 to a persistently exciting input should be sufficient to predict
 the response to \textbf{any} input. 

 Persistence of excitation is essential for system identification
 and adaptive control. Normally, in system identification the
 system of interest is tested only for one input sequence.
 One of the reason for this is that our notion of the system
 entails a fixed initial state. However, any experiment
 changes that particular initial state and it is in general not
 clear how to reset the system to a particular initial state.
 The objective is to find a system model based on
 the response to the chosen input. However, the knowledge of a model
 of the system immediately implies that the response of the
 system to \textbf{any} input is known. Hence, intuitively it is
 clear that persistence of excitation of the input signal is a 
 prerequisite for a successful identification of a model.
 
 Note that persistence of excitation is a joint property of
 the input and of the input-output map. That is, a particular
 input might be persistently exciting for a particular system and
 might fail to be persistently exciting for another system.
 In fact, it is not a priori clear if any system admits 
 a persistently exciting input. 
 This calls for investigating classes of inputs which are
 persistently exciting for some broad classes of systems. 
 
 In the existing literature,
 persistence of excitation is often defined as a specific property
 of the measurements which is  \textbf{sufficient} for
 the correctness of \textbf{some}
 identification algorithm. In contrast, in this paper we propose
 a definition of persistence of excitation which is \textbf{necessary}
 for the correctness of \textbf{any} identification algorithm.
 \footnote{In fact, we also propose a specific algorithm for the
   correctness of which persistence of excitation is sufficient,
   but we do not claim this is true for all identification algorithms.}
 Obviously, the two approaches are complementary. In fact, we hope
 that the results of this paper can serve as a starting point to
 derive persistence of excitation conditions for 
 specific identification algorithms.

 \textbf{Contribution of the paper }
  We define persistence of excitation for finite input sequences
  and persistence of excitation for infinite
  input sequences. 

  We show that for every input-output map
  which is realizable by a reversible \LSS, there exists a finite 
  input sequence which is persistently exciting for that 
  particular input-output map. A reversible \LSS\ is a \LSS\ 
  continuous dynamics of which is invertible. Such systems arise
  naturally by sampling continuous-time systems.
  In addition, we define the class of
  reversible input-output maps and show that there is a 
  finite input sequence which is persistently exciting for all
  the input-output maps of that class. Moreover, we present a
  procedure for constructing such an input sequence.

  We show that 
  there exists a class of infinite input sequences which are 
  persistently exciting for all the input-output maps which are 
  realizable by a \emph{stable \LSS}. The conditions which the input
  sequence must satisfy is that each subsequence occurs there infinitely often (i.e. the switching signal is rich enough) and that
  the continuous input is a colored noise. Hence, this result
  is consistent with the classical result for linear systems.

  It might be appealing to interpret the conditions above as 
  ones which ensure that one stays in every discrete mode long
  enough and the continuous input is persistently exciting in the
  classical sense. One could then try to identify the linear subsystems
  separately and merge the results. Unfortunately, such an 
  interpretation is in general incorrect. The reason for this is
  that there exists a broad class of input-output maps which
  can be realized by a linear switched system but not by a 
  switched system whose linear subsystems are minimal, 
  \cite{MPLBJH:Real}.
  The above scheme obviously would not work for such systems.
  In fact, for such systems one has to test the system's response not
  only for each discrete mode, but for each combination of discrete
  modes.  

  The main idea behind the definition of persistence of excitation
  and the subsequent results is as follows. From realization 
  theory \cite{MPLBJH:Real} we know that the knowledge of
  (finitely many) Markov-parameters of the input-output map
  is sufficient for computing a \LSS\ realization of that map.
  Hence, if the response of the input-output map to a 
  particular input allows us to compute the necessary Markov-parameters, then we can compute a \LSS\ representation of that map.
  This can serve as a definition of persistence of excitation.
  We call a input sequence persistently exciting, if
  the Markov-parameters of the input-output map can be computed from
  the response of the map to that input. We call an infinite
  sequence input persistently exciting, if from a large enough
  finite initial part of the response one can compute an arbitrarily
  precise approximation of the Markov-parameters.
  Since the realization algorithm for \LSS\ is continuous in
  the Markov-parameters, it means that a
  persistently exciting infinite input sequence
  allows the computation of an
  arbitrarily precise approximation of a \LSS\ realizing the input-output
  map.

\textbf{Motivation of the system class}
  The class of \BSLSS is the simplest and perhaps the best studied
  class of hybrid systems. In addition to its practical relevance,
  it also serves as a convenient starting point for theoretical
  investigations. In particular, any piecewise-affine hybrid 
  system can be viewed as a feedback interconnection of a 
  \LSS\ with an event generating device. Hence, identification
  of a piecewise-affine system is related to the problem of
  closed-loop identification of a \LSS. For the latter, it 
  is indispensable to have a good notion of persistence of excitation.
  For this reason, we believe that the results of the paper will be
  relevant not only for identification of \SLSS, but also for
  identification of piecewise-affine hybrid systems with
  autonomous switching.

\textbf{Related work}
  Identification of hybrid systems is an active research
  area, with several significant contributions
  \cite{JLSVidal,MaVidal,Vidal2,Vidal1,IdentSurvey,VidalAutomatica,IdentComparison,PLClustering,Juloski1,Juloski2,Vidal2,Bako09-SYSID2,Bako08-IJC,Roll:Automatica04,Fox:09,Verdult04,Paoletti2,Bako10}.
  While enormous progress was made in terms of efficient
  identification algorithms, the fundamental theoretical
  limitations and properties of these algorithms are still
  only partially understood.
  Persistence of excitation of
  hybrid systems were already addressed in
  \cite{JLSVidal,VidalAutomatica,VidalObs,Verdult04,Hiskens1}.
  However, the conditions of those papers
   are more method-specific and their approach is
   quite different from the one we propose. 
   For linear systems, persistence of excitation has thouroughly
   been investigated, see for example \cite{LjungBook,Willems2005325} 
   and the references therein.

\textbf{Outline of the paper}
  \S \ref{sect:switch} presents the formal definition 
  of \BSLSS and it formulates the major system-theoretic 
  concepts for this system class.
 \S \ref{sect:real} presents a brief overview of
 realization theory for \SLSS.
 \S \ref{sect:pers} presents the main contribution of
 the paper.

\textbf{Notation}
  Denote by $\mathbb{N}$ the set of natural numbers including $0$.
 The notation described below is standard in automata theory, see \cite{GecsPeak,AutoEilen}.
 Consider a set $X$ which will be called the \emph{alphabet}.
 Denote by $X^{*}$ the set of finite
 sequences of elements of $X$.
 Finite sequences of elements of
 $X$ are  referred to as \emph{strings} or \emph{words} over $X$.
 Each non-empty word $w$ is of the form $w=a_{1}a_{2} \cdots a_{k}$
 for some $a_1,a_2,\ldots,a_k \in X$.
 The element $a_i$ is called the \emph{$i$th letter of $w$}, for
 $i=1,\ldots,k$ and $k$ is called the \emph{length of $w$}.
 We denote by $\epsilon$ the \emph{empty sequence (word)}.
 The length of word $w$ is denoted by $|w|$; note that $|\epsilon|=0$.
 We denote by $X^{+}$ the set
 of non-empty words, i.e.
 $X^{+}=X^{*}\setminus \{\epsilon\}$.
 We denote by $wv$ the concatenation of word $w \in X^{*}$ with $v \in X^{*}$.
For each $j=1,\ldots,m$, $e_j$ is the $j$th unit vector of $\mathbb{R}^{m}$, i.e.
  $e_j=(\delta_{1,j},\ldots, \delta_{n,j})$,
  $\delta_{i,j}$ is the Kronecker symbol. 

\section{Linear switched systems}
\label{sect:switch}
 In this section we present the formal definition of \BSLSS along with a number of relevant system-theoretic concepts for \BSLSS. 
 \begin{Definition}
  \label{switch:def}
  Recall from \cite{MP:HSCC2010} that a discrete-time
  linear switched system (abbreviated by \LSS), is a discrete-time
  control system of the form
  \begin{equation}
  \label{lin_switch0}
  \Sigma\left\{
  \begin{array}{lcl}
   x_{t+1} &=& A_{q_t}x_t+B_{q_t}u_t \mbox{ and $x_0=0$}  \\
   y_t  &=& C_{q_t}x_t. 
  \end{array}\right.
  \end{equation}
  Here $Q=\{1,\ldots,\QNUM\}$ is the finite set of discrete modes, 
  $\QNUM$ is a positive integer, 
 $q_t \in Q$ is the
  switching signal, $u_t \in \mathbb{R}$ is the continuous input,
  $y_t \in \mathbb{R}^{p}$ is the output and 
  $A_{q} \in \mathbb{R}^{n \times n}$,
  $B_{q} \in \mathbb{R}^{n \times m}$, $C_q \in \mathbb{R}^{p \times n}$
  are the matrices of the linear system in mode $q \in Q$.
\end{Definition}
 Throughout the section, \emph{$\Sigma$ denotes a \LSS\ of the form (\ref{lin_switch0}).}
The \emph{inputs of $\Sigma$ are the
continuous inputs $\{u_t\}_{t=0}^{\infty}$ and the
switching signal $\{q_t\}_{t=0}^{\infty}$.} The state of the system
at time $t$ is $x_t$.
Note \emph{that any switching
signal is admissible and that the initial state is assumed
to be zero}.
We use the following notation for the inputs of $\Sigma$.
\begin{Notation}[Hybrid inputs]
 Denote $\HYBINP=Q \times \mathbb{R}^{m}$.
\end{Notation}
 We denote by  $\HYBINP^{*}$ (resp. $\HYBINP^{+}$) 
 the set of all 
  finite 
 (resp. non-empty and finite)
sequences
 of elements of $\HYBINP$.
 A sequence 
 \begin{equation}
 \label{inp_seq}
 w=(q_0,u_0)\cdots (q_t,u_t) \in \HYBINP^{+} \mbox{, } t \ge 0 
 \end{equation}
  describes the scenario, when the discrete mode $q_i$ and
  the continuous input $u_i$ are fed to $\Sigma$ at
 time $i$, for $i=0,\ldots,t$.
    \begin{Definition}[State and output]
     Consider a state $x_{init} \in \mathbb{R}^{n}$.
     For any $w \in \HYBINP^{+}$ of the form (\ref{inp_seq}),
     denote by \( x_{\Sigma}(x_{init},w) \)
     \emph{the state of $\Sigma$}
     at time $t+1$, and denote by
     $y_{\Sigma}(x_{init},w)$ the \emph{output} of $\Sigma$
     at time $t$, if $\Sigma$ is started from $x_{init}$ and 
     the inputs $\{u_i\}_{i=0}^t$ and
     the discrete modes $\{q_i\}_{i=0}^{t}$ are fed to the system. 
  \end{Definition}    
     That is,
     $x_{\Sigma}(x_{init},w)$ is defined recursively as follows;
     $x_{\Sigma}(x_{init},\epsilon)=x_{init}$,  and if $w=v(q,u)$
     for some $(q,u) \in \HYBINP$, $v \in \HYBINP^{*}$, then
     \[ x_{\Sigma}(x_{init},w)=A_{q}x_{\Sigma}(x_{init},v)+B_{q}u. \]
     If $w \in \HYBINP^{+}$ and
      $w=v(q,u)$, $(q,u) \in \HYBINP$, $v \in \HYBINP^{*}$,
      then 
      \[ y_{\Sigma}(x_{init},w)=C_{q}x_{\Sigma}(x_{init},v). \]
  \begin{Definition}[Input-output map]
   The map  $y_{\Sigma}: \HYBINP^{+} \rightarrow \mathbb{R}^{p}$,
   $\forall w \in \HYBINP^{+}: y_{\Sigma}(w)=y(x_0,w)$, is called
   the input-output map of $\Sigma$.
  \end{Definition}   
   That is, the input-output map of $\Sigma$ maps each
   sequence $w \in \HYBINP^{+}$ to the output generated
   by $\Sigma$ under the hybrid input $w$, if started from the 
    zero initial state.
   The definition above implies that the input-output behavior
   of a \BLSS can be formalized as a map
   \begin{equation}
   \label{io_map}
     f:\HYBINP^{+} \rightarrow \mathbb{R}^{p}. 
   \end{equation}
   The value $f(w)$ for $w$ of the form \eqref{inp_seq} represents the output 
   of the underlying black-box system at time $t$, 
   if the continuous inputs $\{u_i\}_{i=0}^{t}$ and the switching sequence 
   $\{q_i\}_{i=0}^t$ are fed to the system.

   Next, we define
   when a general map $f$ of the form \eqref{io_map} 
   is adequately described by the
   \BLSS $\Sigma$, i.e. when $\Sigma$ is a realization of $f$.
   \begin{Definition}[Realization]
   \label{switch_sys:real:def1}
    The \BLSS $\Sigma$ is a \emph{realization} of
    an input-output map $f$ of the form \eqref{io_map}, if
   $f$ equals the input-output map of $\Sigma$, i.e. $f=y_{\Sigma}$.
   \end{Definition}
   For the notions of observability and span-reachability of
  \SLSS\ we refer the reader to \cite{MPLBJH:Real,Sun:Book}.
   \begin{Definition}[Dimension]
   \label{switch_sys:dim:def}
    The dimension of $\Sigma$, denoted by $\dim \Sigma$, is 
    the dimension $n$ of its state-space.
   \end{Definition}
\begin{Definition}[Minimality]
 Let 
 $f$ be an
 input-output map.
 Then $\Sigma$ is \emph{a minimal realization of $f$}, if
 $\Sigma$ is a realization of $f$, and
 for any \LSS\ 
 $\hat{\Sigma}$ which is a realization 
 of $f$, $\dim \Sigma \le \dim \hat{\Sigma}$.
\end{Definition}

 

\section{Overview of realization theory}
\label{sect:real}
 Below we present an overview of the 
 results on realization theory
 of \SLSS\ along with the concept of Markov-parameters.
 For more details on the topic see \cite{MPLBJH:Real}.
In the sequel, \emph{$\Sigma$ denotes a \LSS\ of the form
(\ref{lin_switch0}), and $f$ denotes an input-output map
$f:\HYBINP^{+} \rightarrow \mathbb{R}^{p}$.}

For our purposes the most important result is the one
which states that a \LSS\ realization of $f$ can be
computed from the Markov-parameters of $f$. In order to
present this result, we need to define the Markov-paramaters
of $f$ formally.
 Denote $Q^{k,*}=\{ w \in Q^{*} \mid |w| \ge k\}$.
 Define the maps 
 $S_j^f:Q^{2,*} \rightarrow \mathbb{R}^{p}$, $j=1,\ldots,m$
  as follows; for any $v=\sigma_1\ldots\sigma_{|v|} \in Q^{*}$ with $\sigma_k\in Q$, and for any $q,q_0\in Q$,
\begin{equation}
\label{eq:MarkovParam}
	 S^f_{j}(q_0vq)=
	\left\{\begin{aligned}
	  & f\big((q_0,e_j)(q,0)\big) \mbox{ if } v=\epsilon\\
		&f\big((q_0,e_j) (\sigma_1,0) \ldots (\sigma_{|v|},0) (q,0)\big) \mbox{ if } |v|\geq 1
	\end{aligned}\right.
	\end{equation}
with $e_j\in \Re^m$ is the vector with $1$ as its $j$th entry and zero everywhere else. The collection of maps $\{S_j^f\}_{j=1}^{m}$ is 
 called the \emph{Markov-parameters} of $f$.\\
The functions $S_j^f$, $j=1,\ldots,m$ can be viewed as
\emph{input responses}. The interpretation of $S_j^f$
will become more clear 
 after we define
the concept of a \emph{generalized convolution representation}.
Note that the values of the Markov-parameters can be obtained 
from the values of $f$.
 \begin{Notation}[Sub-word]
\label{not:subword}
  Consider the sequence $v=q_0\cdots q_t \in Q^{+}$,
  $q_0,\ldots, q_t \in Q$, $t \ge 0$.
 For each $j,k \in \{0,\ldots,t\}$, define
 the word $v_{j|k} \in Q^{*}$ as follows; 
 if $j> k$, then $v_{j|k}=\epsilon$, if $j=k$, then
 $v_{j|j}=q_j$ and if
 $j < k$, then $v_{j|k}=q_jq_{j+1}\cdots q_k$.
 That is, $v_{j|k}$ is the sub-word of $v$ formed by the letters
 from the $j$th to the $k$th letter.
 \end{Notation}
 \begin{Definition}[Convolution representation]
 \label{sect:io:def1}
  The input-output map $f$ has a \emph{generalized convolution
  representation (abbreviated as \GCR)}, if for all 
  $w \in \HYBINP^{+}$ of the form \eqref{inp_seq},
  $f(w)$ can be expressed via the Markov-parameters of $f$ as follows.
  \begin{equation*}
    \begin{split}
     & f(w)= 
      \sum_{k=0}^{t-1} S^f(q_k \cdot v_{k+1|t-1} \cdot q_t)u_k
    \end{split}
  \end{equation*}
where $S^f(w)=\big[\begin{matrix}S_1^f(w) & \ldots & S_m^f(w) \end{matrix}\big] \in \Re^{p\times m}$ for all $w \in Q^{*}$.
 \end{Definition}
 \begin{Remark}
\label{rem:Unicity-of-from-MP}
  If $f$ has a \GCR, then
  the Markov-parameters of $f$ determine $f$ uniquely.
 \end{Remark}
  The motivation for introducing \GCR{s} is that existence of a \GCR\ is
  a necessary condition for realizability by \SLSS. 
  Moreover, if $f$ is realizable by a \LSS, then  
  the Markov-parameters of $f$ can be expressed as 
  products of the matrices of its \LSS\ realization.
  In order to formulate this result more precisely, we need the
  following notation.
 \begin{Notation}
 \label{repr:not1}
   Consider the collection of $n \times n$ matrices
   $A_{\sigma}$, $\sigma \in X$. For any $w \in Q^{*}$,
   the $n \times n$ matrix $A_w$ is defined as follows.
   If $w=\epsilon$, then $A_{\epsilon}$ is the identity matrix.
   If  $w=\sigma_1\sigma_2 \cdots \sigma_{k} \in X^{*}$, 
   $\sigma_1, \cdots \sigma_k \in X$, $k > 0$, 
  then 
  \begin{equation} 
  \label{repr:eq2}
  A_{w}=A_{\sigma_{k}}A_{\sigma_{k-1}} \cdots A_{\sigma_{1}}. 
  \end{equation}
\end{Notation}
 \begin{Lemma}
 \label{sect:io:lemma1}
  The map $f$ is  realized by the
  \LSS\  $\Sigma$ if and only if $f$ has a \GCR\ and 
  for all $v \in Q^{*}$, $q,q_0 \in Q$,
  \begin{equation}
  \label{sect:io:lemma1:eq1}
	\begin{split}
     & S^f_{j}(q_0vq)=C_qA_vB_{q_0}e_j, \: j=1,\ldots,m. 
      \end{split}
  \end{equation}
 \end{Lemma}

Next, we define the concept of a Hankel-matrix. 
Similarly to the linear case, the entries of
the Hankel-matrix are formed by the Markov parameters. 
For the definition of the Hankel-matrix of $f$,
we will use lexicographical ordering on the
set of sequences $Q^{*}$.
\begin{Remark}[Lexicographic ordering]
\label{rem:lex}
 Recall that $Q=\{1,\ldots,\QNUM\}$. We define a lexicographic
 ordering $\prec$ on $Q^{*}$ as follows.
 For any $v,s \in Q^{*}$, $v \prec s$ if either
 $|v| < |s|$ or $0 < |v|=|s|$, $v \ne s$ and for some $l \in \{1,\ldots, |s|\}$,
 $v_l < s_l$ with the usual ordering of integers and
 $v_i=s_i$ for $i=1,\ldots, l-1$. Here $v_i$ and $s_i$ denote the $i$th letter
 of $v$ and $s$ respectively.
 Note 
 that $\prec$ is a complete ordering and
 $Q^{*}=\{v_1,v_2,\ldots \}$ with $v_1 \prec v_2 \prec \ldots $.
 Note that $v_1 =\epsilon$ and for all $i \in \mathbb{N}$, $q \in Q$,
 $v_i \prec v_iq$.
\end{Remark}
In order to simplify the definition of a Hankel-matrix,
we introduce the notion of a combined Markov-parameter.
\begin{Definition}[Combined Markov-parameters]
 A combined Markov-parameter $M^f(v)$ of $f$ indexed by 
 the word
 $v \in Q^{*}$ is the following
 $p\QNUM \times \QNUM m$ matrix 
 \begin{equation}
   \label{main_results:lin:arb:pow2}
     M^f(v) = \begin{bmatrix}
              S^f(1v1), & \cdots, & S^f(\QNUM v1) \\
              S^f(1v2), & \cdots, & S^f(\QNUM v2) \\
              \vdots     & \cdots & \vdots \\
              S^f(1v\QNUM), & \cdots, & S^f(\QNUM v \QNUM)
             \end{bmatrix}
\end{equation}
\end{Definition}
\begin{Definition}[Hankel-matrix] 
\label{main_result:lin:hank:arb:def}
 Consider the lexicographic ordering $\prec$ of $Q^{*}$ from
 Remark \ref{rem:lex}.
 Define the Hankel-matrix $H_f$ of $f$ as the following
 infinite matrix
 \[ H_f = \begin{bmatrix}
     M^f(v_1v_1) & M^f(v_2v_1) & \cdots & M^f(v_kv_1) & \cdots \\
     M^f(v_1v_2) & M^f(v_2v_2) & \cdots & M^f(v_{k}v_2) & \cdots \\
     M^f(v_1v_3)
     & M^f(v_2v_3) & \cdots & M^f(v_{k}v_3) & \cdots \\
     \vdots   
      & \vdots   & \cdots & \vdots & \cdots 
   \end{bmatrix},
 \]
i.e. the $p\QNUM \times (m\QNUM)$ 
block of $H_f$    
in the block row $i$ and block column $j$
equals the combined Markov-parameter $M^f(v_jv_i)$ of $f$.
The rank of $H_f$, denoted by $\Rank H_f$, 
is the dimension of the linear span of its columns.
\end{Definition}
The main result on realization theory of
\SLSS\ can be stated as follows.
  \begin{Theorem}[\cite{MPLBJH:Real}]
  \label{sect:real:theo2}
   \begin{enumerate}
   \item
      The map $f$ has a realization by  a \LSS\  if and
      only if $f$ has a \GCR\ 
      and 
     $\Rank H_{f} < +\infty$. 
  \item
  A minimal \LSS\ realization of $f$
     can be constructed from $H_{f}$
     and any minimal \LSS\ 
     realization of $f$ has dimension $\Rank H_f$.
  \item
     A \LSS\  $\Sigma$ is a minimal realization of $f$ if 
     and only if $\Sigma$ is span-reachable, observable and
    it is a realization of $f$. Any two \SLSS\ which are
    minimal realizations of $f$ are isomorphic \footnote{see
    \cite{MPLBJH:Real} for the definition of isomorphism between
    \SLSS}.
  \end{enumerate}
\end{Theorem}

Note that Theorem \ref{sect:real:theo2}
 shows that the knowledge of
the Markov-parameters is necessary and sufficient for
finding a state-space representation of $f$. In fact,
similarly to the continuous-time case \cite{MP:PartTechRep}, 
we can even show that the knowledge of \emph{finitely many}  
Markov-parameters is sufficient.  This will be done 
by formulating a \emph{realization algorithm}
for \SLSS, which computes a \SLSS\ realization of $f$
based on finitely many Markov-parameters of $f$. 

In order to present the realization algorithm, we need the 
following notation.

 \begin{Notation}
  Consider the lexicographic ordering $\prec$ of $Q^{*}$
  and recall that $Q^{*}=\{v_1,v_2,\ldots, \}$ where
  $v_1 \prec v_2 \cdots $. Denote by $\mathbf{N}(L)$
  the number of sequences from $Q^{*}$ of length 
  at most $L$. It then follows that 
  $|v_i| \le L$ if and only if $i \le \mathbf{N}(L)$.
 \end{Notation}
 \begin{Definition}[$H_{f,L,M}$ sub-matrices of $H_{f}$]
 \label{sect:main_result:sub-hank:def1}
  For $L,K \in \mathbb{N}$ 
  define the integers 
  $I_L=\mathbf{N}(L)p\QNUM$ and $J_K=\mathbf{N}(K)m\QNUM$ 
  Denote by $H_{f,L,K}$ the following upper-left $I_L \times J_K$ 
  sub-matrix of $H_{f}$, 
 \[ 
     \begin{bmatrix}
     M^f(v_1v_1) & M^f(v_2v_1) & \cdots & M^f(v_{\mathbf{N}(K)}v_1) \\
     M^f(v_1v_2) & M^f(v_2v_2) & \cdots & M^f(v_{\mathbf{N}(K)}v_2) \\
     \vdots   
      & \vdots   & \cdots & \vdots & \\
     M^f(v_1v_{\mathbf{N}(L)}) & M^f(v_2v_{\mathbf{N}(L)}) & \cdots & M^f(v_{\mathbf{N}(K)}v_{\mathbf{N}(L)}) 
   \end{bmatrix}.
 \]
 \end{Definition}
   Notice that the entries of $H_{f,L,K}$ are
   Markov-parameters indexed by words of length at most $L+K$, i.e.
   $H_{f,L,K}$ is uniquely determined by $\{M^f(v_i)\}_{i=1}^{\mathbf{N}(L+K)}$.
The promised realization algorithm is Algorithm \ref{alg0},
which takes as input the matrix $H_{f,N,N+1}$ and produces a 
\LSS. Note that the knowledge of $H_{f,N,N+1}$ is equivalent
to the knowledge of the 
finite sequence $\{M^f(v_i)\}_{i=1}^{\mathbf{N}(2N+1)}$ of
Markov-parameters. The correctness of Algorithm \ref{alg0}
is stated below.

  \begin{Theorem}
  \label{part_real_lin:theo1}
  If $\Rank H_{f,N,N}=\Rank H_{f}$, then 
  Algorithm \ref{alg0} returns 
  a minimal realization $\Sigma_N$ of $f$.
  The condition $\Rank H_{f,N,N}=\Rank H_{f}$ holds for a given $N$, if
  there exists a \LSS\  realization $\Sigma$ of $f$ such that
  $\dim \Sigma \le N+1$.
\end{Theorem}
 The proof of Theorem \ref{part_real_lin:theo1} 
 is completely analogous to
 its continuous-time counterpart \cite{MP:PartTechRep}.
 \emph{Theorem \ref{part_real_lin:theo1} implies that if 
 $f$ is realizable by a \LSS, then 
  a minimal \LSS\  realization of
  $f$  is computable from 
 finitely many Markov-parameters}, using Algorithm \ref{alg0}.
 In fact, if $f$ is realizable by a \LSS\ of dimension $n$, 
 then the \emph{first $\mathbf{N}{(2n-1)}$ Markov-parameters
 $\{M^f(v_i)\}_{i=1}^{\mathbf{N}(2n-1)}$ uniquely 
 determines $f$}. 

   \begin{algorithm}
   \caption{
    \newline
    \textbf{Inputs:} Hankel-matrix $H_{f,N,N+1}$.  
    \newline
    \textbf{Output:} \LSS\ $\Sigma_N$
    }
    \label{alg0} 
    \begin{algorithmic}[1]
    \STATE
       Let $n=\Rank H_{f,N,N+1}$. Choose a tuple
       of integers $(i_1,\ldots,i_n)$ such that
       the columns of $H_{f,N,N+1}$ indexed by $i_1,\ldots,i_n$
       form a basis of $\IM H_{f,N,N+1}$.
       Let $\mathbf{O}$ be 
       $I_N \times n$ matrix formed by these
       linearly independent columns, 
       i.e. the $r$th column  of $\mathbf{O}$ equals
       the $i_r$th column of $H_{f,N,N+1}$.
        Let 
       $\mathbf{R} \in \mathbb{R}^{n \times J_{N+1}}$ be the 
       matrix, $r$th column of which is formed by 
       coordinates of the $r$th column of $H_{f,N,N+1}$
       with respect to the basis consisting of the
       columns $i_1,\ldots,i_n$ of $H_{f,N,N+1}$, for every
       $r=1,\ldots,J_{N+1}$. It then follows that
       $H_{f,N,N+1} = \mathbf{O}\mathbf{R}$ and
       $\Rank \mathbf{R} =\Rank \mathbf{O} = n$.
   \STATE
     Define $\bar{\mathbf{R}} \in \mathbb{R}^{n \times J_N}$ as the
     matrix formed by the first $J_N$ columns of $\mathbf{R}$.

  \STATE
     For each $q \in Q$, let $\mathbf{R}_q \in \mathbb{R}^{n \times J_N}$ 
     be such that for each $i=1,\ldots J_N$, the $i$th column of
     $\mathbf{R}_q$ equals the $r(i)$th column of $\mathbf{R}$. Here  
     $r(i) \in \{1,\ldots,J_{N+1}\}$ is defined as follows.
     Consider the decomposition $i=(r-1)m\QNUM+z$ 
     for some $z=1,\ldots,m\QNUM$ and
     $r=1,\ldots,\mathbf{N}(N)$.
     Consider
     the word $v_rq$ and notice that
    $|v_rq| \le N+1$. Hence, $v_rq=v_{d}$
     for some $d=1,\ldots,\mathbf{N}(N+1)$. Then define
     $r(i)$ as $r(i)=(d-1)m\QNUM+z$.

    \STATE
       
       Construct $\Sigma_N$ of the form (\ref{lin_switch0}) such that
          \begin{eqnarray}
              & & 
		\begin{bmatrix} B_1,\ldots, B_{\QNUM}
                \end{bmatrix} = \nonumber  \\
             & & \mbox{the first $m\QNUM$ columns of $\mathbf{R}$} 
          \label{alg0:eq-2} \\
             & &  \begin{bmatrix} C_1^T & C_2^T & \ldots & C_{\QNUM}^T 
\end{bmatrix}^T = \mbox{ the first $p\QNUM$ rows of $\mathbf{O}$} 
             \label{alg0:eq-1}  \\
	& & \forall q \in Q:  A_q=\mathbf{R}_{q}\mathbf{\bar{R}}^{+} 
         \label{alg0:eq1}
          \end{eqnarray}
      where $\mathbf{\bar{R}}^{+}$ is the Moore-Penrose pseudoinverse
      of $\mathbf{\bar{R}}$.
    \STATE
       Return $\Sigma_N$
    \end{algorithmic}
   \end{algorithm}
  The intuition behind Algorithm \ref{alg0}
  is the following. The state-space of the \LSS\ 
  $\Sigma_N$ returned by Algorithm \ref{alg0} is an
  isomorphic copy of the space spanned by the
  columns of $H_{f,N,N}$. The isomorphism is
  determined by the matrix $\mathbf{R}$. 
  The columns of $B_q$, $q \in Q$ are formed by the
  columns $(q-1)m\QNUM+1,\ldots, qm\QNUM$ of the block-matrix 
  \[ \begin{bmatrix} M^f(v_1v_1)^T & \ldots & M^f(v_1v_{\mathbf{N}(L)})^T \end{bmatrix}^T.\]
   The rows of $C_q$, $q \in Q$ 
   are formed by
   the rows $(q-1)p+1,\ldots,pq$ of $H_{f,N,N+1}$.
   Finally, the matrix $A_q$, $q \in Q$ is the matrix
   of a shift-like operator, which maps a block-column 
   $\left\{M^f(v_jv_i)\right\}_{i=1}^{\mathbf{N}(L)}$ of $H_{f,N,N}$ to
   the block-column 
   $\left\{M^f(v_jqv_i)\right\}_{i=1}^{\mathbf{N}(L)}$ of $H_{f,N,N+1}$.

\section{Main results of the paper}
\label{sect:pers}
  The main idea behind our definition of persistence
  of excitation is as follows. The measured time series
  is persistently exciting, if from this time-series
  we can reconstruct the Markov-parameters of the underlying
  system. Note that by Theorem \ref{part_real_lin:theo1}, it is enough to 
  reconstruct \emph{finitely many} Markov-parameters.
  This also means that our definition of persistence of
  excitation is also applicable to finite time series.

  In order to present our main results, we will need some
  terminology.
  \begin{Definition}[Output time-series]
    For any input-output map $f$ and for any
    finite input sequence $w \in \HYBINP^{+}$
   we denote by $\mathbf{O}(f,w)$
   the output time series induced by $f$ and $w$, i.e.
   if $w$ is of the form \eqref{inp_seq}, then 
   $\mathbf{O}(f,w)=\{y_t\}_{t=0}^{T}$, such that
   $y_t=f((q_0,u_0)\cdots (q_t,u_t))$ for all $t \le T$.
  \end{Definition}
  \begin{Definition}[Persistence of excitation]
   The finite sequence $w \in \HYBINP^{+}$ is 
   \emph{persistently exciting for the input-output map $f$}, if
   it is possible to determine 
   the Markov-parameters
   of $f$ from the data $(w,\mathbf{O}(f,w))$.
  \end{Definition}
  \begin{Remark}[Interpretation]
  \label{interp:rem1}
  Theorem \ref{part_real_lin:theo1} and Algorithm \ref{alg0} allow the following
  interpretation of persistence of excitation defined above.
  If $w$ is persistently exciting, then the Markov-parameters
  of $f$ can be computed from the response of $f$ to
  the prefixes of $w$. In particular,   
  if $f$ admits a \LSS\ realization of dimension at most
  $n$, then 
  the Markov-parameters $\{M^f(v_i)\}_{i=1}^{\mathbf{N}(2n-1)}$ 
  can be computed from the data $(w,\mathbf{O}(f,w))$. 
  The knowledge of $\{M^f(v_i)\}_{i=1}^{\mathbf{N}(2n-1)}$ is 
  sufficient for computing a \LSS\ realization of $f$.
  Hence, \emph{persistence of excitation of $w$
  for $f$ means that Algorithm \ref{alg0} can serve as an 
  identification algorithm for computing a \LSS\ realization
  of $f$ from the time-series $(w,\mathbf{O}(f,w))$}.
  Note, however, that our definition does not depend
  on Algorithm \ref{alg0}. Indeed, 
  if there is any algorithm which can correctly find a 
  \LSS\ realization of $f$ from $(w,\mathbf{O}(f,w))$, then
  according to our definition, $w$ is persistently exciting.
  Note that our definition of persistence of excitation involves
  only the inputs, but not the output response. 
  \end{Remark}

   So far we have defined the persistence of excitation for finite
   sequences of inputs. 
   Next, we define the same notion for infinite sequences
   of inputs. To this end, we need the following notation.
   \begin{Notation}
    We denote by $\HYBINP^{\omega}$ the set of infinite sequences
    of hybrid inputs. That is, any element $w \in \HYBINP^{\omega}$
    can be interpreted as a time-series
    $w=\left\{(q_t,u_t)\right\}_{t=0}^{\infty}$. For each $N \in \mathbb{N}$,
    denote by $w_N$ the sequence formed by the first $N$ elements
    of $w$, i.e. $w_N=(q_0,u_0)\cdots (q_N,u_N)$.
  \end{Notation}
  \begin{Definition}[Asymptotic persistence of excitation]
   An infinite sequence of inputs $w \in \HYBINP^{\omega}$ is called
   \emph{asymptotically persistently exciting for
    the input-output map $f$}, if the following holds. For 
    every sufficiently large $N$, we can compute
    from $(w_N, \mathbf{O}(f,w_N))$ asymptotic 
    estimates of the Markov-parameters of $f$.
    More precisely, for  $N \in \mathbb{N}$, we can compute from 
    $(w_N,\mathbf{O}(f,w_N))$ 
    some matrices
    $\{M^f_N(v)\}_{v \in Q^{*}}$ such that 
    $\lim_{N \rightarrow \infty} M^f_N(v)=M^f(v)$ for
    all $v \in Q^{*}$. 
    When clear from the context, we will use the term
    persistently exciting instead of asymptotically persistently
    exciting.
  \end{Definition}



\begin{Remark}[Interpretation]
\label{rem:interp2}
  The interpretation of asymptotic persistence of
   excitation is that \emph{asymptotically persistently
  exciting inputs allow us to estimate a \LSS\ realization
  of $f$ with arbitrary accuracy}. Indeed, assume that
  $w \in \HYBINP^{\omega}$ is
  asymptotically persistently exciting. Then for
  each $N$ we can compute from
  the time-series $(w_N,\mathbf{O}(f,w_N))$  an
  approximation $\{M_N^f(v)\}_{v \in Q^{*}}$
  of the Markov-parameters of $f$.
  Suppose that $f$ is realizable by a \LSS\ of dimension
  $n$ and we know the indices $(i_1,\ldots,i_n)$
  of those columns of $H_{f,n-1,n}$ which form a basis
  of the column space of $H_{f,n-1,n}$.
  Let $H^{N}_{f,n-1,n}$ be the matrix which is constructed
  in the same way as $H_{f,n-1,n}$, but with $M_N^f(v)$ instead
  of the Markov-parameters $M^f(v)$.  Since $M_N^f(v)$ converges
  to $M^{f}(v)$ for all $v \in Q^{*}$, we get that
  each entry of $H^N_{f,n-1,n}$ converges to the corresponding entry
  of $H_{f,n-1,n}$. Modify Algorithm \ref{alg0}
  by fixing the choice of columns to $(i_1,\ldots,i_n)$ in the first
  step. It is easy to see that
 the modified algorithm represents a continuous map from
 the input data (finite Hankel-matrix) to the output data
 (matrices of a \LSS). For sufficiently large $N$,
  the columns of $H_{f,n-1,n}^N$ indexed by $(i_1,\ldots,i_n)$
  also represent a basis of the column space of $H_{f,n-1,n}^N$.
  If we apply the modified Algorithm \ref{alg0} to
  the sequence of matrices $H^N_{f,n-1,n}$, 
  we obtain a sequence of \SLSS\ $\Sigma_{n,N}$ and
  the parameters of $\Sigma_{n,N}$ converge 
  to the parameters of the \LSS\ $\Sigma$
  which we would
  obtain from Algorithm \ref{alg0} if we applied it to
  $H_{f,n-1,n}$. In particular, by choosing a 
  sufficiently large $N$, the
  parameters of $\Sigma_{n,N}$ are sufficiently close to
  those of $\Sigma$.
 \end{Remark}

  We will show that for \textbf{every} reversible \LSS\ 
  there exists \textbf{some} 
  input which is persistently exciting.
  In addition, we present a class of inputs which are 
  persistently exciting of any input-output map $f$
  realizable by a stable \LSS.

  \subsection{Persistently exciting input for
              specific systems}

  \label{sec:spec_pers}
  In this section we present results which state that
  for any input-output map $f$ which is realizable
  by a \LSS, there exists a persistently exciting 
  finite input. 

  Note that from \eqref{eq:MarkovParam} it follows that the
  Markov-parameters of $f$ can be obtained from finitely many
  input-output data. However, the application of \eqref{eq:MarkovParam}
  implies evaluating the response of the system for different
  inputs, while started from a fixed initial state.
  In order to simulate this by evaluating the response of
  the system to one single input (which is then necessarily
  persistently exciting), one has to provide means to
  reset the system to its initial state. 
  In order to be able
  to do so, we restrict attention to \emph{reversible}
  \SLSS.
  \begin{Definition}
  A \LSS\ $\Sigma$ of the form \eqref{lin_switch0} is
  \emph{reversible}, if for every discrete mode $q \in Q$,
   the matrix $A_q$ is invertible.
  \end{Definition}
  Reversible \SLSS\ arise naturally when sampling continuous-time
  systems.
  \begin{Theorem}
  \label{pers:theo1}
   Consider an input-output map $f$. Assume that $f$ has a 
   realization by a reversible \LSS.
   Then there exists an input $w \in \HYBINP^{+}$ such that
   $w$ is persistently exciting for $f$. 
  \end{Theorem}
  \begin{proof}[Sketch of the proof]
   The main idea behind the proof of Theorem \ref{pers:theo1}
   is as follows.  If $f$ admits a \LSS\ realization
   of dimension $n$, then the finite sequence
   $\{M^f(v_i)\}_{i=1}^{\mathbf{N}(2n-1)}$ of
   Markov-parameters determine all the Markov-parameters 
   of $f$ uniquely. Hence, in order for a finite
   input $w$ to be persistently exciting for $f$, it is
   sufficient that $\{M^f(v_i)\}_{i=1}^{\mathbf{N}(2n-1)}$
   can be computed from the response $(w,\mathbf{O}(f,w))$.

   Note that \eqref{eq:MarkovParam} implies that 
   $\{M^f(v_i)\}_{i=1}^{\mathbf{N}(2n-1)}$ 
   can be computed from the responses 
   of $f$ from finitely many inputs. More precisely,
   $\{M^f(v_i)\}_{i=1}^{\mathbf{N}(2n-1)}$ can
   be computed from $\{f(s) \mid s \in S\}$,
   where 
\[
     \begin{aligned}
     & S=\{(q_0,e_j) (\sigma_1,0)\ldots (\sigma_{|v_i|},0) (q,0) \in \HYBINP^{+} \mid q_0,q \in Q, \\
		&v_i=\sigma_1\ldots \sigma_{|v_i|},  j=1,\ldots,m, i=1,\ldots, \mathbf{N}(2n-1) \}.
     \end{aligned}
  \]	
   Hence, if for each $s \in S$ there exists a prefix $p$ 
   of $w$ such that $f(s)=f(p)$, then this $w$ will be
   persistently exciting. 

   One way to construct such a 
   $w$ is to construct for each $s \in S$ an input
   $s^{-1} \in \HYBINP^{+}$ such that 
   \[ \forall v \in \HYBINP^{+}: f(ss^{-1}v)=f(v). \]
   That is, the input $s^{-1}$ neutralizes the effect of
   the input $s$. We defer the construction of the
   input $s^{-1}$ to the end of the proof.
   Assume for the moment being that such  inputs
   $s^{-1}$ exist.
   Let $S=\{s_1,\ldots,s_d\}$ be an
   enumeration of $S$. Then it is easy to see that
   $f(s_1s_1^{-1}s_2)=f(s_2)$,
   $f(s_1s_1^{-1}s_2s_2^{-1}s_3)=f(s_3)$, etc. 
   Hence, if we define 
   \[ w=s_1s_1^{-1}\cdots s_{d-1}s_{d-1}^{-1}s_d, \]
  then each $f(s)$, $s \in S$ can be obtained as a 
  response of $f$ to a suitable prefix of $w$.
  Hence,  $w$ is persistently exciting.

  It is left to show that $s^{-1}$ exists. 
  Consider a reversible realization $\Sigma$ of $f$.
   Then the controllable set and reachable set of $\Sigma$
   coincide by \cite{Sun:Disc}.
   Hence, from any reachable state $x$ of
   $\Sigma$, there exists an input $w(x)$ such that
   $w(x)$ drives $\Sigma$ from $x$ to zero, i.e.
   $x_{\Sigma}(x,w(x))=0$. For each $s \in S$,
   let $x(s)=x_{\Sigma}(0,s)$ and define
   $s^{-1}=w(x(s))$ as the input which drives $x(s)$
   back to the initial zero state.
 \end{proof}
 It is easy to see that Theorem \ref{pers:theo1} can be extended to
  any input-output map which admits a controllable \LSS\ realization.
  However, it is not clear if every input-output map
  which is realizable by a \LSS\ is also realizable by a
  controllable \LSS.
  Note that the construction of the persistently exciting
  $w$ from Theorem \ref{pers:theo1} requires the knowledge
  of a \LSS\ realization of $f$. 
   Below we present a subclass of input-output maps, for which
   the knowledge of a state-space representation is
   not required to construct a persistently
   exciting input.
  \begin{Definition}
  \label{spec:rev}
   Fix a map $.^{-1}:\HYBINP \ni \alpha \mapsto \alpha^{-1} \in \HYBINP$.
   A input-output map $f$ is said to be \emph{reversible with
   respect to the map $.^{-1}$}, if 
   for all $\alpha \in \HYBINP$, $s,w \in \HYBINP^{*}$, $|sw| > 0$,
   \[ f(s\alpha\alpha^{-1}w)=f(sw). \]
  \end{Definition}
   Intuitively, $f$ is reversible with respect to $.^{-1}$, if the effect of any input $\alpha=(q,u)$
   can be neutralized by the input $\alpha^{-1}$.
   Such a property is not that uncommon, think for example of
   turning a valve on and off. 
   For example,
   if $f$ has a realization by a \LSS\ $\Sigma$ of the form 
   \eqref{lin_switch0},
   and $Q=\{1,\ldots, 2K\}$ such that for each $q \in \left\{1,\ldots,K\right\}$,
   $A_q=A_{q+K}^{-1}$, $B_{q}=-AB_{q+K}$, then $f$ is reversible and 
   $(q,u)^{-1}=(q+K,-u)$.
	
   \noindent From the proof of Theorem \ref{pers:theo1}, we obtain the following
   corollary.
   \begin{Theorem}
   \label{spec:rev:theo}
   If $f$ is reversible with respect to $.^{-1}$, 
  then a persistently exciting 
   input sequence $w$ can be constructed for $f$. The construction
   does not require the knowledge of a \LSS\ state-space realization 
   of $f$. If the inputs $\alpha^{-1}$ 
   from Definition \ref{spec:rev} are
   computable from $\alpha$, then the construction of $w$ is
   effective.
   \end{Theorem} 
   \begin{proof}[Proof of Theorem \ref{spec:rev:theo}]
    The proof differs from that of Theorem \ref{pers:theo1}
    only in the definition of $s^{-1}$ for each
    $s \in S$. More precisely, if $f$ is reversible,
    then for each $s=(q_0,u_0)\cdots (q_t,u_t) \in S$ 
    define
    \[ s^{-1}=(q_t,u_t)^{-1}(q_{t-1},u_{t-1})^{-1}\cdots (q_0,u_0)^{-1}
    \]
   \end{proof}

 \subsection{Universal persistently exciting inputs}
 \label{sec:uni}  
   
 Next, we discuss classes of inputs which are 
 persistently exciting for all input-output maps
 realizable by \SLSS.

 \begin{Definition}[Persistence of excitation cond]
 \label{uni_pe:def}
  An
 infinite input $w=\{(q_t,u_t)\}_{t=0}^{\infty} \in \HYBINP^{\omega}$ satisfies PE condition,
  if there exists a strictly positive definite $m \times m$
  matrix $\mathcal{R}$ and a collection of strictly positive numbers
  $\{\pi_{v}\}_{v \in Q^{+}}$ such that for any positive integer
  $j \in \mathbb{N}$ and any word $v \in Q^{+}$
  \[ 
     \begin{split}
      & \lim_{N \rightarrow \infty} \frac{1}{N} \sum_{t=0}^{N}
       u_{t+j}u_t^T\chi(q_tq_{t+1}\cdots q_{t+|v|-1}=v) = 0 \\
      & \lim_{N \rightarrow \infty} \frac{1}{N} \sum_{t=j}^{N}
       u_{t-j}u_t^T\chi(q_{t-j}q_{t-j+1}\cdots q_{t-j+|v|-1}=v) = 0 \\
      & \lim_{N \rightarrow \infty} \frac{1}{N} \sum_{t=0}^{N}
          u_{t}u_{t}^{T}\chi(q_t\cdots q_{t+|v|-1}=v)=\pi_{v}\mathcal{R}
     \end{split}
  \]
  where $\chi$ is the indicator function, i.e. $\chi(A)=1$ if $A$
  holds and $\chi(A)=0$ otherwise.
 \end{Definition}
 \begin{Remark}[PE condition implies rich switching]
 Note that if $w \in \HYBINP^{\omega}$ satisfies the conditions
 of Definition \ref{uni_pe:def}, then
 \[ \lim_{N \rightarrow \infty} \sum_{t=0}^{N} \chi(q_t\cdots q_{t+|v|-1}=v)=\pi_{v} > 0
 \]
 for each $v \in Q^{+}$. This implies that  any sequence of
 discrete modes occurs in the switching signal. Hence, our condition
 for persistence of excitation implies that the switching signal
 should be rich enough. This is consistent with many of the existing
 definitions of persistence of excitation for hybrid systems.
 The requirement that $\pi_v > 0$ for all $v \in Q^{*}$ is quite 
 a strong one. At the end of this section we will discuss
 possible relaxations of this requirement.
 \end{Remark}
 \begin{Remark}[Relationship with stochastic processes]
 \label{uni_pe:def:rem}
  Fix a probability space $(\Omega,\mathcal{F},P)$ and consider
  ergodic discrete-time stochastic processes 
  $\textbf{u}_t:\Omega \rightarrow \mathbb{R}^{m}$ and
  $\textbf{q}_t:\Omega \rightarrow  Q$
  with values in $\mathbb{R}^m$ and $Q$ respectively.
  In addition, assume the following.
  \begin{itemize}
  \item
  The processes $\textbf{u}_t$ and $\textbf{q}_t$ are independent (i.e. the $\sigma$-algebras generated
  by $\{\textbf{u}_t\}_{t=0}^{\infty}$ and by $\{\textbf{q}_t\}_{t=0}^{\infty}$ are independent.
  \item
     The stochastic process $\textbf{u}_t$ is a colored noise, i.e. it is zero-mean, 
    $\textbf{u}_t$ and $\textbf{u}_s$ are uncorrelated and $E[\textbf{u}_t\textbf{u}_t^T]=\mathcal{R}> 0$, with $E[\cdot]$ denoting the expectation operator. 
  \item
        For each $v \in Q^{+}$,
        $\pi_{v}=P(\textbf{q}_t\cdots \textbf{q}_{t+|v|-1}=v) > 0$.
  \end{itemize}
 It then follows that almost all sample paths of $\textbf{u}_t$,
 $\textbf{q}_t$ satisfy the PE condition of Definition \ref{uni_pe:def}. That is, there exists a set $A \in \mathcal{F}$, such that
  $P(A)=0$ and for all $\omega \in \Omega \setminus A$,
  the sequence $w=\{(q_t,u_t)=(\textbf{q}_t(\omega),\textbf{u}_t(\omega)\}_{t=0}^{\infty}$ satisfies the PE condition.
 \end{Remark}
 \begin{Remark}
 \label{uni_pe:def2}
  If $\textbf{u}_t$ is a white-noise Gaussian process and if the
  variables
  $\textbf{q}_{t}$ are uniformly distributed over $Q$ 
  (i.e. $P(\textbf{q}_t=q)=\frac{1}{|Q|}$ and are independent from
  each other and from $\{\textbf{u}_{s}\}_{s=0}^{\infty}$,
  then $\textbf{u}_t$ and $\textbf{q}_t$ satisfy
  the conditions of Remark \ref{uni_pe:def:rem} and hence
  almost any sample path of $\textbf{u}_t$ and $\textbf{q}_t$ 
  satisfies the PE condition of Definition \ref{uni_pe:def}.

  This special case also provides a simple practical way to
  generate inputs which satisfy the PE conditions.
 \end{Remark}
  
 We will show that input sequences which satisfy
 the conditions of Definition \ref{uni_pe:def} are
 asymptotically persistently exciting for a large class
 of input-output maps.
 The main idea behind the theorem is as follows.
 Consider a \LSS\ $\Sigma$ which is realization of $f$, 
 and suppose we feed
 a stochastic input $\{\textbf{q}_t, \textbf{u}_t\}$ 
 into $\Sigma$. Then the state $\textbf{x}_t$ 
 and the output response 
 $\textbf{y}_t$ of $\Sigma$ will also be stochastic processes.
 Suppose that $\{\textbf{q}_t,\textbf{u}_t\}$ are
 stochastic processes which satisfy the
 conditions of Remark \ref{uni_pe:def:rem}.
 It is easy to see that
 \[ \textbf{y}_{t}=\sum_{k=0}^{t} C_{\textbf{q}_t}A_{\textbf{q}_{t-1}}\cdots A_{\textbf{q}_{k+1}}B_{\textbf{q}_{k}}
 \textbf{u}_k. \]
 and hence for all $r,q \in Q$, $v \in Q^{*}$, $|rvq|=t+1$,
 \begin{equation}
 \label{uni:eq1}
   \begin{split}
    & E[\textbf{y}_{t}\textbf{u}_0^T\chi(\textbf{q}_0\cdots \textbf{q}_{t}=rvq)]=  \\
    & \sum_{k=0}^{t} C_{q}A_{v}B_{r}E[\textbf{u}_k\textbf{u}_0^T\chi(\textbf{q}_0\cdots \textbf{q}_{t}=rvq)]= \\
    & C_{q}A_{v}B_{r}\mathcal{R}\pi_{rvq}=S^f(rvq)\mathcal{R}\pi_{rvq}.
   \end{split}
 \end{equation}
 Hence, if we know the expectations
 $E[\textbf{y}_{t}\textbf{u}_0^T\chi(\textbf{q}_0\cdots \textbf{q}_{t}=rvq)]$ for all $r,q \in Q$, $v \in Q^{*}$, 
  $|rvq|=t+1$, $t > 0$, then we can find all the Markov-parameters of $f$, by the following formula
\[ 
   S^f(rvq)=E[\textbf{y}_{t}\textbf{u}_0^T\chi(\textbf{q}_0\cdots \textbf{q}_{t+1}=rvq)]\mathcal{R}^{-1}\frac{1}{\pi_{rvq}}.
\]

Hence, the problem of estimating the Markov-parameters
reduces to estimating the expectations
\begin{equation}  
\label{uni:eq2}
  E[\textbf{y}_{t}\textbf{u}_0^T
   \chi(\textbf{q}_0\cdots \textbf{q}_{t}=rvq)].
\end{equation}
For practical purposes, the expectations in
\eqref{uni:eq2} have to be estimated from a sample-path
of $\textbf{y}_t$, $\textbf{u}_t$ and $\textbf{q}_t$.
The most natural way to accomplish this is to use
the formula
\begin{equation}
\label{uni:eq3} 
   \lim_{N \rightarrow \infty} \frac{1}{N}
   \sum_{t=i}^{N} y_{i+t}u_i^T\chi(q_{i}\cdots q_{i+t}=rvq) 
\end{equation}
where $y_t$, $u_t$, $q_t$ denote the value at time $t$ of a
sample-path of $\textbf{y}_t$, $\textbf{u}_t$ and
$\textbf{q}_t$ respectively. Note that 
$y_t$ is in fact the output of $\Sigma$ at time $t$, if
the input $\{u_i\}_{i=0}^{t}$ and the switching signal $\{q_i\}_{i=0}^{t}$ are fed to the system.

The problem with estimating \eqref{uni:eq2} by \eqref{uni:eq3}
is that the limit \eqref{uni:eq3} may fail to exist or to
converge to \eqref{uni:eq2}. 

A particular case when 
\eqref{uni:eq3} converges to \eqref{uni:eq2} is when the
process $(\textbf{y}_t,\textbf{u}_t,\textbf{q}_t)$ is ergodic.
In that case, we can choose a sample-path
$(y_t,u_t,q_t)$ of $(\textbf{y}_t,\textbf{u}_t,\textbf{q}_t)$
for which the limit in \eqref{uni:eq3} equals the expectation 
\eqref{uni:eq2} ; in fact `almost all' sample paths will have this
property. This means that we can choose a suitable deterministic
input sequence $\{u_{t}\}_{t=0}^{\infty}$ and a switching signal
$\{q_t\}_{t=0}^{\infty}$, such that for the resulting output
$\{y_t\}_{t=0}^{\infty}$, the limit \eqref{uni:eq3} equals the
expectation \eqref{uni:eq2}. That is, in that case the
input $w=(q_0,u_0)\cdots (q_t,u_t) \cdots $ is asymptotically
persistently exciting. However, proving ergodicity of
$\textbf{y}_t$ is not easy. In addition, even if $\textbf{y}_t$
is ergodic, the particular choice of the deterministic
input $w$ for which 
\eqref{uni:eq3} equals \eqref{uni:eq2} might depend on the 
\LSS\ itself.

For this reason, instead of using the concepts of ergodicity
directly, we just show that for the input sequences $w$ which
satisfy the conditions of Definition \ref{uni_pe:def},
the corresponding output $\{y_t\}_{t=0}^{\infty}$ has the 
property that the limit \eqref{uni:eq3} exists and it equals
$S^f(rvq)\mathcal{R}\pi_{rvq}$, for any input-output map $f$ which is realizable
by a \emph{$l_1$-stable \LSS}. This strategy allows us to use elementary
techniques, while not compromising the practical relevance of
the result.

In order to present the main result of this section, we have to
define the notion of \( l_1\)-stability of \SLSS. 

\begin{Definition}[Stability of \SLSS]
 A \LSS\ $\Sigma$ of the form \eqref{lin_switch0} is called
 $l_1$-stable, if for every $x \in \mathbb{R}^{n}$,
 the series $\sum_{v \in Q^{*}} ||A_vx||_{2}$ is convergent.
\end{Definition}
\begin{Remark}[Sufficient condition for stability]
 If for all $q \in Q$, $||A_q||_{2} < \frac{1}{|Q|}$, where
 $||A_q||_{2}$ is the matrix norm of $A_q$ induced by the
 standard Euclidean norm, then $\Sigma$ is $l_1$-stable.
\end{Remark}
\begin{Remark}[Asymptotic stability]
 If $\Sigma$ is $l_1$-stable, 
  then it is asymptotically stable, in the sense
 that if $s_i \in Q^{*}$, $i > 0$ is a sequence of words such that
 $\lim_{i \rightarrow \infty} |s_i|=+\infty$, then
 $\lim_{i \rightarrow \infty} A_{s_i}x=0$ for all $x \in \mathbb{R}^n$.
\end{Remark}
 Intuitively it is clear why we have to restrict attention to
 stable systems. Recall that \eqref{eq:MarkovParam} allows us
 to compute the Markov-parameters of $f$ from the responses of
 $f$ to finitely many inputs. In order to obtain the response
 of $f$ to several inputs from the response of $f$ to 
 one input, one has to find means to suppress the contribution of
 the current state of the system to future inputs.
 In \S  \ref{sec:spec_pers} this was done by feeding inputs which drive the
 system back to the initial state. Unfortunately, the choice of such
 inputs depended on the system itself. By assuming stability, we
 can make sure that the effect of the past state will 
 asymptotically diminish in time. Hence, by waiting long enough,
 we can approximately recover the response of $f$ to any input.

 Another intuitive explanation for assuming stability is that
 it is necessary for the stationarity, and hence ergodicity,
 of the output and state processes $\textbf{y}_t$, $\textbf{x}_t$.

 Equipped with the definitions above, we can finally state the
 main result of the section.
 \begin{Theorem}[Main result]
 \label{uni:theo}
  If $w$ satisfies the PE conditions of 
  Definition \ref{uni_pe:def}, then
  $w$ is asymptotically persistently exciting for any 
  input-output map $f$ which admits a $l_1$-stable 
  \LSS\ realization.
 \end{Theorem}
 The theorem above together with Remark \ref{uni_pe:def2}
 imply that white noise input and a binary noise switching signal
 are asymptotically persistently exciting.
  The proof of Theorem \ref{uni:theo}
  relies on the following technical result.
 \begin{Theorem}
 \label{main:lemma}
  Assume that
  $\Sigma$ is a $l_1$-stable
  \LSS\ of the form \eqref{lin_switch0}, and
  assume that $w$ satisfies the PE conditions.
  Let $\{y_t\}_{t=0}^{\infty}$ and $\{x_t\}_{t=0}^{\infty}$ 
  be the output and state response of $\Sigma$
  to $w$, i.e. $y_t=y_{\Sigma}(w_t)$ and
  $x_t=x_{\Sigma}(0,w_t)$. Then for all $v,\beta \in Q^{*}$,
  $r,q \in Q$
  \begin{eqnarray}
    & \pi_{rvq\beta}A_vB_{r}\mathcal{R}= \nonumber \\
    & \lim_{N \rightarrow \infty} \frac{1}{N}
    \sum_{t=0}^{N} x_{t+|v|+1}u^T_t\chi(t,rvq\beta) 
   \label{main:lemma:eq1} \\
    & \pi_{rvq\beta} C_qA_vB_{r}\mathcal{R}= \nonumber \\
    & \lim_{N \rightarrow \infty} \frac{1}{N}
    \sum_{t=0}^{N} y_{t+|v|+1}u^T_t\chi(t,rvq\beta) 
   \label{main:lemma:eq2} 
  \end{eqnarray}
  Here we used the following notation: for all $s \in Q^{+}$,
  \[
     \chi(t,s)=\left\{\begin{array}{rl}
                    1 & \mbox{ if } s=q_tq_{t+1}\cdots q_{t+|s|-1} \\
                    0 & \mbox{ otherwise }
                \end{array}\right.
  \]
 \end{Theorem}
 Informally, Theorem \ref{main:lemma} implies that if
 $f$ is realizable by a $l_1$-stable \LSS, then 
 the limit \eqref{uni:eq3} equals \eqref{uni:eq2}.
 The proof of Theorem \ref{main:lemma} can be found in 
 Appendix \ref{proofs}.

 \begin{proof}[Proof of Theorem \ref{uni:theo}]
  For each $t$, denote by $y_t$ the response of $f$ to the
  first $t$ elements of $w$, i.e.
  $y_t=f((q_0,u_0)\cdots (q_t,u_t))$.
  For each integer $N \in \mathbb{N}$ and for each
  word $v \in Q^{*}$, define the matrix $S_{N}(rvq)$ as
  \[ S_{N}(rvq) = (\frac{1}{N} \sum_{t=0}^{N} y_{t+|v|+1}u^T_t\chi(t,rvq))\mathcal{R}^{-1}\frac{1}{\pi_{rvq}}
   \]
  and define the matrix $M_N(v)$ by
  \[
     \begin{bmatrix}
     S_N(1v1) & \cdots & S_N(\QNUM v1)  \\
     \vdots       & \vdots & \vdots \\
     S_N(1v \QNUM) & \cdots & S_{N}(\QNUM v\QNUM)  \\
     \end{bmatrix}
  \]
 From Theorem \ref{main:lemma} it follows that 
 \[ \lim_{N \rightarrow \infty} S_{N}(rvq)=S^f(rvq) \]
 and hence $\lim_{N \rightarrow \infty} M_{N}(v)=M^f(v)$.
 Hence, $w$ is indeed asymptotically persistently exciting.
 \end{proof}
 \begin{Remark}[Relaxation of PE condition]
  Assume that we restrict attention to input-output maps which
  are realizable by a $l_1$-stable \LSS\ of dimension at most $n$,
  and let $f$ be such an input-output map.
  In this case, one can replace the conditions of  
  Definition \ref{uni_pe:def}, that $\pi_v > 0$ by the condition
  that $\pi_{s} > 0$ for all $|s| \le 2n+1$ and still obtain
  asymptotically persistently exciting inputs for $f$.
 
  Indeed, consider now any $w \in \HYBINP^{\omega}$ which satisfies Definition \ref{uni_pe:def} with
  the exception that $\pi_{v} > 0$ is required only for 
  $|v| \le 2n+1$. Then Theorem 
  \ref{main:lemma} remains valid for this case (the proof remains
   literally the same) and from the proof of  Theorem \ref{uni:theo}
   we get that for all $i=1,\ldots, \mathbf{N}(2n-1)$,
  \[
      S^f(rv_iq)=\lim_{N \rightarrow \infty} (\frac{1}{N} \sum_{t=0}^{N} y_{t+|v|+1}u^T_t\chi(t,rv_iq))\mathcal{R}^{-1}\frac{1}{\pi_{rv_iq}}
  \]
  Hence, $\{M^f(v_i)\}_{i=1}^{\mathbf{N}(2n-1)}$ can asymptotically
  be estimated
  from $(w_N,\mathbf{O}(f,w_N))$.
  Since the modified Algorithm \ref{alg0} from Remark \ref{rem:interp2}
  determines a continuous map from
  $\{M^f(v_i)\}_{i=1}^{\mathbf{N}(2n-1)}$ to
  the other Markov-parameters of $f$,  
  $w$ is asymptotically persistently exciting for $f$.
 \end{Remark}

\section{Conclusions}
  We defined persistence of excitation for
  input signals of linear switched systems. We showed
  existence of persistently exciting input sequences
  and we identified several classes of input
  signals which are persistently exciting. 

  Future work includes finding less restrictive
  conditions for persistence of excitation and
  extending the obtained results to other classes of
  hybrid systems. 
	

\appendix
\section{Technical proofs}
\label{proofs}
 The proof of Theorem \ref{main:lemma} relies on
 the following result.
 \begin{Lemma}
 \label{main:basic:lemma}
  With the notation and assumptions of Theorem
  \ref{main:lemma}, for all $v \in Q^{+}$, 
  \[ \lim_{N \rightarrow \infty} \frac{1}{N} \sum_{t=0}^{N} x_{t}u^T_t\chi(t,v) =0 
   \]
 \end{Lemma}
 The intuition behind Lemma \ref{main:basic:lemma} is as follows.
 Each $x_t$ is a linear combination of inputs $u_{0},\ldots,u_{t-1}$.
 Hence, $\frac{1}{N} \sum_{t=0}^{N} x_{t}u^T_t$ can be
 expressed as linear combination of terms
 $\frac{1}{N} \sum_{t=k}^{N} u_{t-k}u_t^{T}\chi(t,s)$ for some
 $s \in Q^{*}$, $k=1,\ldots,N$. Since each such term converges
 to $0$ as $N \rightarrow \infty$, intuitively their
  linear combination should converge to $0$ as well.
 Unfortunately, the number of summands of the above 
 increases with $N$. In order to deal
 with this difficulty a technique similar to the $M$-test for
 double series has to be used.
 The assumption that $\Sigma$ is $l_1$-stable is 
 required for this technique to work.
 \begin{proof}[Proof of Theorem \ref{main:lemma}]
  We start with the proof of \eqref{main:lemma:eq1}.
  The proof goes by induction on the length of $v$.
  
  If $v=\epsilon$, then 
  \begin{equation}
  \label{main:lemma:pf:eq1}
   \begin{split}
    & \frac{1}{N} \sum_{t=0}^{N} x_{t+1}u_t^T\chi(t,r\beta)=  \\
    & \frac{1}{N} \sum_{t=0}^{N} (A_{q_t}x_{t}+B_{q_t}u_{t})u^T_t\chi(t,r\beta)= \\
    & \frac{1}{N} \sum_{t=0}^{N} A_{q_t}x_{t}u^T_t\chi(t,r\beta) + 
      \frac{1}{N} \sum_{t=0}^{N} B_{q_t}u_{t}u^T_t\chi(t,r\beta). 
   \end{split}
  \end{equation}
  Notice $A_{q_t}x_tu^T_t\chi(t,r\beta)=A_rx_tu^T_t\chi(t,r\beta)$ and
  $B_{q_t}u_tu_t^T\chi(t,r\beta)=B_ru_tu_t^T\chi(t,r\beta)$.
 Hence,
  \begin{equation}
  \label{main:lemma:pf:eq2}
   \begin{split}
    & \frac{1}{N} \sum_{t=0}^{N} A_{q_t}x_{t}u^T_t\chi(t,r\beta) +
      \frac{1}{N} \sum_{t=0}^{N} B_{q_{t}}u_{t}u^T_t\chi(t,r\beta)= \\
   & A_r(\frac{1}{N} \sum_{t=0}^{N} x_tu_t^T\chi(t,r\beta))+
     B_r(\frac{1}{N} \sum_{t=0}^{N} u_tu_t^T\chi(t,r\beta))
   \end{split}
  \end{equation}
 From the assumptions on $w$ it follows that
 \[ 
  \begin{split}
  & \lim_{N \rightarrow \infty} \frac{1}{N} \sum_{t=0}^{N} u_tu_t^T\chi(t,r\beta)=\mathcal{R}\pi_{r\beta}
  \end{split}
 \]
 Hence, from the PE conditions 
and Lemma \ref{main:basic:lemma} we get that 
  \[ 
    \begin{split}
     &  \lim_{N \rightarrow \infty}
        \frac{1}{N} \sum_{t=0}^{N} x_{t+1}u_t^T \chi(t,r\beta) = \\
     & A_r(\lim_{N \rightarrow \infty} \frac{1}{N} \sum_{t=0}^{n} x_tu_t^T\chi(t,r\beta))+ \\
       &+ B_r(\lim_{N \rightarrow \infty} \frac{1}{N} \sum_{t=0}^{n} u_tu_t^T\chi(t,r\beta)) = \\
     & A_r0+B_r \mathcal{R} \pi_{r\beta}=\pi_{r\beta}B_r \mathcal{R},
    \end{split}
  \]
 i.e. \eqref{main:lemma:eq1} holds. 

 Assume that \eqref{main:lemma:eq1} holds for all words
 of length at most $L$, and assume that $v=wq$, $|w|=L$
 for some $w \in Q^{*}$ and $q \in Q$. Then
 by the induction hypothesis and the assumptions on $w$
  \begin{equation}
  \label{main:lemma:pf:eq3}
   \begin{split}
    & \lim_{N \rightarrow \infty } \frac{1}{N} \sum_{t=0}^{N} x_{t+L+2}u_t^T\chi(t,rwq\beta)= \\
    & \lim_{N \rightarrow \infty} \frac{1}{N} \sum_{t=0}^{N} A_{q}x_{t+L+1}u^T_t\chi(t,rwq\beta) +  \\
      & +\lim_{N \rightarrow \infty} \frac{1}{N} \sum_{t=0}^{N} B_{q}u_{t+L+1}u^T_t\chi(t,rwq\beta)= \\
  &  =A_qA_wB_r\pi_{rwq\beta}+B_r0=A_{wq}B_r\mathcal{R}\pi_{rwq\beta}.
   \end{split}
  \end{equation}
 Finally, we prove \eqref{main:lemma:eq2}. Notice that
 \[ y_{t+|v|+2}u_t^T\chi(q,t,rvq\beta)=C_qx_{t+|v|+2}u_t^T\chi(t,rvq\beta) \]
 and hence by applying \eqref{main:lemma:eq1},
 \[  
   \begin{split}
    & \lim_{N \rightarrow \infty} \frac{1}{N} \sum_{t=0}^{N} y_{t+|v|+2}u_t^T\chi(t,rvq\beta)      = \\
   & C_q\lim_{N \rightarrow \infty} \frac{1}{N} \sum_{t=0}^{N} x_{t+|v|+2}u_t^T\chi(t,rvq\beta)=\\
   & C_qA_{v}B_{r} \mathcal{R}\pi_{rvq\beta}.
   \end{split}
 \]
 \end{proof}
 \begin{proof}[Proof of Lemma \ref{main:basic:lemma}]
  Notice that
  \[ 
    \begin{split}
     & \sum_{t=1}^{N} x_tu_t^T\chi(t,v)= \\
       & \sum_{t=1}^{N} \sum_{j=1}^{t-1} A_{q_{t-1}}\cdots A_{q_j}
  B_{q_{j-1}}u_{j-1}u_t^T\chi(t,v) = \\
     & \sum_{k=1}^{N-1} (\sum_{t=k}^{N} A_{q_{t-1}}\cdots A_{q_{t-k+1}}B_{q_{t-k}}u_{t-k}u_t^T\chi(t,v))= \\
     & \sum_{r \in Q} \sum_{k=0}^{N-1} \sum_{|s|=k}
       A_{s}B_{r} \sum_{t=k+1}^{N} u_{t-k-1}u_t^T\chi(t-k-1,rsv)= \\
     &\sum_{i=0}^{\mathbf{N}(N)} \sum_{r \in Q} A_{v_i}B_r
      \sum_{t=|v_i|+1}^{N} u_{t-|v_i|-1}u_t^T\chi(t-|v_i|-1,rv_iv).
    \end{split}
  \]
  In the last step we used the lexicographic ordering 
  of $Q^{*}$ from Remark \ref{rem:lex}.
  It then follows that
  \[
     \begin{split}
      & \frac{1}{N}  \sum_{t=1}^{N} x_tu_t^T\chi(t,v)= \\
      & \sum_{r \in Q} \sum_{i=0}^{\mathbf{N}(N)} A_{v_i}B_r
      \frac{1}{N} \sum_{t=|v_i|+1}^{N} u_{t-|v_i|-1}u_t^T\chi(t-|v_i|-1,rv_iv).
     \end{split}
  \] 
  Define
  \[ 
    \begin{split}
		& b^r_{i,N}=\frac{1}{N} \sum_{t=|v_i|+1}^{N} u_{t-|v_i|-1}u_t^T\chi(t-|v_i|-1,rv_iv) \\
    & a^r_{i,N} = A_{v_i}B_r b^r_{v_i,N}.
    \end{split}
   \]
  Then the statement of the lemma can be shown by showing that
  for all $r \in Q$, $i=1,2,\ldots$, 
  \[ \lim_{N \rightarrow \infty} \sum_{k=0}^{\mathbf{N}(N)} a^r_{i,N} = 0.  \]
  To this end, notice from the PE conditions that 
  \[ 
    \begin{split}
     & \lim_{N \rightarrow \infty} a^r_{i,N} =  \\
     & A_{v_i}B_r\lim_{N \rightarrow \infty}
          \frac{1}{N} \sum_{t=k+1}^{N} u_{t-k-1}u_{t}^T
                  \chi(t-k-1,rv_iv)=0. 
    \end{split}
  \]
  Moreover, for a fixed $N$ and $k$, we can get the following
  estimate
  \[ 
     ||a^r_{i,N}||_{2} \le ||A_{v_i}B_{r}||_{2}||b_{i,N}^r||_{2}. 
  \]
  If we can show that $||b_{v_i,N}^r||_{2}$ is bounded by a 
  number $K$, then we get that
  \[ ||a^r_{i,N}||_{2} \le ||A_{v_i}B_r||_{2}K. \]
 The latter inequality is already sufficient to finish the
 proof. Indeed, 
  let $D^r_i=||A_{v_i}B_r||_2K$ and notice from the $l_1$-stability assumption on the realization $\Sigma$ that  
 \[ \sum_{i=1}^{\infty} D^r_i=K\sum_{v \in Q^{*}} ||A_vB_r||_2\]
  is convergent.
  Hence,  we get that
   for every $\epsilon > 0$ there exists a 
     $I_{\epsilon}$ such that 
  \[ \sum_{i=I_{\epsilon}+1}^{\infty}  D_i^r < \epsilon/2.
  \]
  For every $N> I_{\epsilon}$,
  \[
    \begin{split}
     & ||\sum_{i=1}^{\mathbf{N}(N)} a^r_{i,N}||_2 =
     ||\sum_{i=1}^{I_{\epsilon}} a^r_{i,N}+\sum_{i=I_{\epsilon}+1}^{\mathbf{N}(N)} a^r_{i,N}||_2 \le  \\
     & \sum_{i=1}^{I_{\epsilon}} ||a^r_{i,N}||_2+
       \sum_{i=I_{\epsilon}+1}^{\mathbf{N}(N)} D^r_i < \sum_{i=1}^{I_{\epsilon}}       ||a^r_{i,N}||_2 + \epsilon/2.
    \end{split}
  \]
  Since $\lim_{N \rightarrow \infty} a^r_{i,N}=0$, there exists
  $N_{\epsilon} \in \mathbb{N}$ such that for all $N > N_{\epsilon}$,
  $||a_{i,N}^{r}||_2 < \frac{\epsilon}{2I_{\epsilon}}$.
  Define $\hat{N}_{\epsilon}$ to be an integer such that
  $\hat{N}_{\epsilon} > N_{\epsilon}$ and $\mathbf{N}(\hat{N}_{\epsilon}) > I_{\epsilon}$. Then for every $N > \hat{N}_{\epsilon}$,
  $\mathbf{N}(N) \ge \mathbf{N}(\hat{N}_{\epsilon}) > I_{\epsilon}$ and 
  \[
\begin{split}
		&	||\sum_{i=1}^{\mathbf{N}(N)} a^r_{i,N}||_{2} \le 
			\sum_{i=1}^{I_{\epsilon}} ||a^r_{i,N}||_2 + \epsilon/2 < \\
		& I_{\epsilon}\frac{\epsilon}{2I_{\epsilon}}+\epsilon/2 =
				\epsilon/2+\epsilon/2=\epsilon.
	\end{split}
  \]
  In other words, $\lim_{N \rightarrow 0} \sum_{i=1}^{\mathbf{N}(N)} a_{i,N}^r=0$.

 It is left to show that $||b_{i,N}^r||_{2} \le K$ for some $K > 0$ and for all $i=1,2,\ldots$, $r \in Q$.
  \begin{equation} \label{lemma:pf1}
   \begin{split}
     & ||b^r_{i,N}||_{2} \le
 \Big\|\frac{1}{N} \sum_{t=|v_i|+1}^{N} u_{t-|v_i|-1}u_{t}^T\chi(t-|v_i|-1,rv_iv)\Big\|_{2} \le \\
 & \Big\|\frac{1}{N} \sum_{t=|v_i|+1}^{N} u_{t-|v_i|-1}u_{t}^T\chi(t-|v_i|-1,rv_iv)\Big\|_{F} = \\
     & 
      \Big[\sum_{i,j=1}^{m} \frac{1}{N^2}\Big\{\sum_{t=|v_i|+1}^{N} 
        (u_{t-|v_i|-1})_i\chi(t-|v_i|-1,rv_iv)(u_t)_j\Big\}^2\Big]^{1/2}.
   \end{split}
  \end{equation}
where $||.||_F$ denotes the matrix Frobenius-norm, 
and $||.||_2$ denotes the matrix norm induced by the Euclidean norm.
 The application of the Cauchy-Schwartz inequality to
 $(\sum_{t=|v_i|+1}^{N} (u_{t-|v_i|-1})_i\chi(t-|v_i|-1,rv_iv)(u_t)_j)^2$
 leads to 
 \begin{equation}
 \label{lemma:pf2} 
 \begin{split}
   & \Big[\sum_{t=|v_i|+1}^{N} (u_{t-|v_i|-1})_i\chi(t-|v_i|-1,rv_iv)(u_t^T)_j\Big]^2
    \le \\
   & \Big(\sum_{t=|v_i|+1}^{N} (u_{t-|v_i|-1})_i^2\chi(t-|v_i|-1,rv_iv)\Big)
    \Big(\sum_{t=|v_i|}^{N} (u_t)^2_j\Big).
  \end{split}
 \end{equation}
 Notice that $(u_{t-|v_i|-1})_i^2\chi(t-|v_i|-1,rv_iv) \le (u_{t-|v_i|-1})_i^2$, since $\chi(t-|v_i|-1,rv_iv) \in [0,1]$. Hence,
 \[ 
   \begin{split}
    & \sum_{t=|v_i|+1}^{N} (u_{t-|v_i|-1})_i^2\chi(t-|v_i|-1,rv_iv)  \le \\
    & \le \sum_{t=|v_i|+1}^{N} (u_{t-|v_i|-1})_i^2 \le \sum_{t=0}^{N} (u_t)_i^2.
   \end{split}
 \]
 Similarly,
  \[ \sum_{t=|v_i|+1}^{N} (u_{t})^2_j \le
      \sum_{t=0}^{N} (u_t)_j^2.
  \]
  Combining these remarks with \eqref{lemma:pf2},
  we obtain
  \begin{equation}
  \label{lemma:pf3}
  \begin{split}
  & \Big[\frac{1}{N^2} \sum_{t=|v_i|+1}^{N} (u_{t-|v_i|-1})_i\chi(t-|v_i|-1,rv_iv)(u_t^T)_j\Big]^2  \\
   & \le \Big(\frac{1}{N}\sum_{t=0}^{N} (u_t)^2_{i}\Big)\Big(\frac{1}{N}\sum_{t=0}^{N} (u_t)_j^2\Big).  
  \end{split}
  \end{equation}
  Notice that
  $\lim_{N \rightarrow \infty } \frac{1}{N} \sum_{t=0}^{N} (u_t)^2_{i}=\mathcal{R}_{ii}$ and hence
  $\frac{1}{N}\sum_{t=0}^{N} (u_t)^2_{i}$ is bounded from above
  by some positive number $K_i$. Using this fact and
  by substituting \eqref{lemma:pf3} into \eqref{lemma:pf1},
  we obtain
  \[ 
     ||b^r_{i,N}||_{2} \le 
     (\sum_{i,j=1}^{m} K_iK_j)^{1/2}.
  \]
  Hence, if we set $K=\sum_{i,j=1}^{m} K_iK_j$, then
  then $||b^r_{i,N}||_{2} \le K$, which is what had to be
  shown.
 \end{proof}


\end{document}